\documentclass[12pt]{amsart}

\usepackage[left=3cm,marginpar=2.5cm,tmargin=2cm,bmargin=2.5cm]{geometry}
\usepackage{thmtools}
\newtheorem{theorem}{Theorem}[section]
\newtheorem{corollary}[theorem]{Corollary}
\newtheorem{lemma}[theorem]{Lemma}

\usepackage{amsfonts, amsthm, amssymb, amsmath}
\usepackage{natbib}
\usepackage{url}
\usepackage{mathtools}
\usepackage{tikz-cd}

\theoremstyle{definition}
\newtheorem{defi}[theorem]{Definition}
\newtheorem{definition}[theorem]{Definition}
\theoremstyle{remark}

\newcommand{\inhyper}[1] {\tilde{#1}}

\def \Sym {\operatorname{Sym}}

\def \Bun {\operatorname{Bun}}

\def \div {\operatorname{div}}

\def \PP {\mathbb P}
\newcommand*{\ra}{\rightarrow}
\newcommand*{\xra}{\xrightarrow}

\begin{document}

\title[Bounds for the stalks of perverse sheaves]{Bounds for the stalks of perverse sheaves in characteristic $p$ and a conjecture of Shende and Tsimerman }
\author{Will Sawin}
\address{Department of Mathematics \\ Columbia University \\ New York, NY 10027}
\email{sawin@math.columbia.edu}

\maketitle

\vspace{-15pt}{\centering\emph{With an appendix by Jacob Tsimerman}

 }
\vspace{15pt}

\begin{abstract} We prove a characteristic $p$ analogue of a result of Massey which bounds the dimensions of the stalks of a perverse sheaf in terms of certain intersection multiplicities of the characteristic cycle of that sheaf. This uses the construction of the characteristic cycle of a perverse sheaf in characteristic $p$ by Saito. We apply this to prove a conjecture of Shende and Tsimerman on the Betti numbers of the intersections of two translates of theta loci in a hyperelliptic Jacobian. This implies a function field analogue of the Michel-Venkatesh mixing conjecture about the equidistribution of CM points on a product of two modular curves. \end{abstract}

\section{Introduction}

Massey used the polar multiplicities of a Lagrangian cycle in the cotangent bundle of a smooth complex manifold to bound the Betti numbers of the stalk of a perverse sheaf at a point \cite[Corollary 5.5]{Massey}. In this paper, we prove an analogous result in characteristic $p$. We use the characteristic cycles for constructible sheaves on varieties of characteristic p defined by \citet[Definition 5.10]{saito1}, building heavily on work of \citet{Beilinson}. Before stating our main theorem, let us define the polar multiplicities. 

\begin{defi} We say a closed subset, or algebraic cycle, on a vector bundle is \emph{conical} if it is invariant under the $\mathbb G_m$ action by dilation of vectors. \end{defi}

\begin{defi} For a vector bundle $V$ on a variety $X$, let $\mathbb P(V) = \operatorname{Proj} (\operatorname{Sym}^* (V^\vee ))$ be its projectivization, whose dimension $\dim X + \operatorname{rank} V-1$, which is equivalent to the quotient of the affine bundle $V$, minus its zero section, by $\mathbb G_m$. For a conical cycle $C$ on $V$, let $\mathbb P(C)$ the quotient of $C$, minus its intersection with the zero section, by $\mathbb G_m$. \end{defi} 

\begin{defi}\label{polar-multiplicity} Let $X$ be a smooth variety of dimension $n$. Let $C$ be a conical cycle on the cotangent bundle $T^* X$ of $X$ of dimension $n$ and let $x$ be a point on $X$.

For $0 \leq i< \dim X$, let $V$ be a sub-bundle of $T^* X$ defined over a neighborhood of $x$, with rank $i+1$. such that the fiber $V_x$ is a general point of the Grassmanian of $i+1$-dimensional subspaces of $(T^* X)_x$. 

Then we define \emph{the $i$th polar multiplicity of $C$ at $x$},  $\gamma^i_C(x)$, as the multiplicity of the pushforward $\pi_* ( \mathbb P(C) \cap \mathbb P(V))$ at $x$, where $\pi: \mathbb P( T^* X) \to X$ is the projection. 

We define {\emph the $n$th polar multiplicity of $C$ at $x$} to be the multiplicity of the zero-section in $C$. \end{defi}

Here $\pi_*( \mathbb P(C) \cap \mathbb P(V))$ is interpreted as an algebraic cycle, and the multiplicity of an algebraic cycle at a point is the appropriate linear combination of the multiplicities of its irreducible components. We will check that this multiplicity is independent of the choice of $V$ with $V_x$ sufficiently general in Section \ref{equivalences} below.

Our result is as follows:

\begin{theorem}\label{first-Massey-intro} Let $X$ be a smooth variety over a perfect field $k$ and let $\ell$ be a prime invertible in $k$. Let $K$ be a perverse sheaf of $\mathbb F_\ell$-modules on $X$. 

Then $\dim_{\mathbb F_\ell} \mathcal H^{-i}(K)_x$ is at most $i$th polar multiplicity of $CC(K)$ at $x$. \end{theorem}

The analogous statement follows for perverse $\ell$-adic sheaves by noting that their Betti numbers are bounded by the Betti numbers of their mod $\ell$ incarnations.

We have a corollary that describes when these Betti numbers must vanish, which may admit a more direct proof:

\begin{corollary}\label{vanishing-intro} Let $X$ be a smooth variety over a perfect field $k$ and let $\ell$ be a prime invertible in $k$. Let $K$ be a perverse sheaf of $\mathbb F_\ell$-modules on $X$. Then $\mathcal H^{-i} (K)_x$ vanishes for \[ -i> \dim SS(K)_x - \dim X \] where $(SS(K))_x$ is the fiber of the singular support of $K$ over $x$. \end{corollary}

Note that the singular support of a perverse sheaf $K$ is simply the support of its characteristic cycle \cite[Proposiiton 5.14(2)]{saito1}.

Our proof follows to a large extent the strategy of \citep{Massey}. In particular, we successively apply nearby and vanishing cycles to reduce to sheaves on lower-dimensional varieties. However, one key difference is the argument of \citep{Massey} involves passing to an analytic neighborhood and applying vanishing cycles along a sufficiently general analytic function. This general function will in particular be transverse to all the strata of a Whitney stratification of $X$ associated to $K$, except possibly the point $x$ itself. In characteristic $p$, we have access to neither analytic neighborhoods, analytic functions, nor well-behaved Whitney stratifications. Instead, we use a general rational function of degree two (i.e. a pencil of conics). We describe in Lemmas \ref{transversality-check} and \ref{nearby-formula} the transversality conditions the map must satisfy for our argument to work, in terms of the characteristic cycle, and then check, in Lemma \ref{genericity-lemma}, that a general pencil of conics satisfies all these transversality conditions. 

A similar approach can hopefully be used to adapt other arguments involving the characteristic cycle from characteristic $0$ to characteristic $p$.  (Hypersurfaces of larger degree would work equally well, and might come in handy if even more transversality is needed.)

\subsection{Application to equidistribution in $\Bun_2(\mathbb P^1)$}

In this paper, we prove, as an application of Theorem \ref{first-Massey-intro}:

Let $k$ be a field of characteristic $\neq 2$ and let $C$ be a hyperelliptic curve of genus $g$ over $k$. Define $\Theta_n$ to be the space of degree $n$ effective divisor classes on $C$, viewed as a closed subscheme of the variety $\operatorname{Pic}^n(C)$ parameterizing degree $n$ divisor classes.

\begin{theorem}\label{theta-betti-bound} For any $g \in \mathbb N$, $0,\leq a,b \leq g$, and $L \in \operatorname{Pic}^{2g-a-b} (C)$, we have \[ \sum_{i \in \mathbb Z} \dim  H^i (  (\Theta_{g-a} \cap L - \Theta_{g-b} )_{\overline{k}}, \mathbb Q_\ell) \leq 28^g/16 + 4\cdot 8^g + 2\cdot 4^g. \]  \end{theorem}

This verifies a conjecture of \citet[Conjecture 1.4]{TS}.

\citet[Theorem 4.4]{TS} proved that this conjecture implies a certain equidistribution result, described below:

Let $\operatorname{Bun}_2(\mathbb P^1)$ be the set of isomorphism classes of rank two vector bundles on $\mathbb P^1_{\mathbb F_q}$, up to tensor products with line bundles on $\mathbb P^1_{\mathbb F_q}$. Let $\operatorname{Bun}_2^0(\mathbb P^1)$ be the subset consisting of rank two vector bundles with even degree, and $\operatorname{Bun}_2^1(\mathbb P^1)$ the subset consisting of bundles with odd degree. (Note that these subsets are stable under tensor product with line bundles).  Both $\operatorname{Bun}_2^0(\mathbb P^1)$ and $\operatorname{Bun}_2^1 (\mathbb P^1)$ admit ``uniform" probability measures $\mu_{ \operatorname{Bun}_2^0(\mathbb P^1)} $ and $\mu_{\operatorname{Bun}_2^1 (\mathbb P^1)}$, where the probability of a vector bundle is proportional to the inverse of the order of its automorphism group.  

Let $C$ by a hyperelliptic curve of genus $g$ over $\mathbb F_q$, with a fixed degree two map $\pi: C \to \mathbb P^1$. For $L$ a line bundle on $C$, $\pi_* L$ is a rank two vector bundle on $\mathbb P^1$, and hence defines a point of $\operatorname{Bun}_2(\mathbb P^1)$. This point is preserved by tensoring $L$ with line bundles pulled back from $\mathbb P^1$, so we can think of $\pi_* L$ as a function from $\operatorname{Pic}(C)/ \operatorname{Pic}( \mathbb P^1) $ to $\operatorname{Bun}_2(\mathbb P^1)$. Because  $\operatorname{Pic}(C)/ \operatorname{Pic}( \mathbb P^1) $ is a finite group, it admits a uniform probability measure.

\begin{theorem}\label{equidistribution-statement} Let $q> 28^4=  614,656$ be a prime power.

Fix a sequence of pairs $(C_i,M_i)$ of hyperelliptic curves $C_i$ and line bundles $M_i$ on $C$. Suppose that $\deg M_i \mod 2$ is constant, $g(C_i)$ converges to $\infty$, and the minimum $n$ such that $M_i$ is equivalent in $\operatorname{Pic}(C)/ \operatorname{Pic}( \mathbb P^1) $ to a divisor of degree $n$ converges to $\infty$ with $i$. 

Then as $i$ goes to $\infty$, the pushforward of the uniform probability measure on  $\operatorname{Pic}(C_i)/ \operatorname{Pic}( \mathbb P^1) $ along the map $ L \mapsto (\pi_* L, \pi_* (L \otimes M_i)) $ from $\operatorname{Pic}(C_i)/ \operatorname{Pic}( \mathbb P^1) $  to $\Bun_2(\mathbb P^1)$ converges to 

\[ \frac{1}{2}  \mu_{ \operatorname{Bun}_2^0(\mathbb P^1)} \times \mu_{ \operatorname{Bun}_2^0(\mathbb P^1)} + \frac{1}{2} \mu_{ \operatorname{Bun}_2^1(\mathbb P^1)}\times \mu_{ \operatorname{Bun}_2^1(\mathbb P^1)}\] if $\deg M_i \mod 2=0$ for all i and \[ \frac{1}{2}  \mu_{ \operatorname{Bun}_2^0(\mathbb P^1)} \times \mu_{ \operatorname{Bun}_2^1(\mathbb P^1)} + \frac{1}{2} \mu_{ \operatorname{Bun}_2^0(\mathbb P^1)}\times \mu_{ \operatorname{Bun}_2^1(\mathbb P^1)}\] if $\deg M_i \mod 2 \neq 0$ for all $i$. \end{theorem} 

This follows immediately from Theorem \ref{theta-betti-bound} and \citep[Theorem 4.4]{TS} (which covers in addition the case where the minimum $n$ does does not converge to $\infty$.) 

We now provide some context for these results:

For an imaginary quadratic number field $K$, we can consider the probability measure on the modular curve $X(1)$ that assigns equal measure to the points corresponding to all elliptic curves with complex multiplication by $\mathcal O_K$. Duke's theorem says that, as the discriminant of the fields go to $\infty$, these measures converge to the uniform measure on $X(1)$ \citep{Duke}.

Recalling that, over the complex numbers, there is a natural bijection between the elliptic curves with complex multiplication by $K$ and the class group $Cl(K)$, for each $\alpha$ in $Cl(K)$, let $z_{K,\alpha}$ be point of $X(1)$ of the elliptic curve corresponding to the class group element $\alpha$. For an ideal class $\sigma$, et $\mu_{K,\sigma}$ be the probability measure on $X(1)$ that assigns equal mass to $(z_{K,\alpha}, z_{K, \sigma\alpha})$ for all $\alpha$ in the class group. (One reason this set of points is natural to consider is that it is an orbit under the Galois group $\operatorname{Gal}(\overline{K}|K)$.)

A generalization of Duke's theorem conjectured by \citet[Conjecture 2 on p. 7]{MVUnpublished} is that $\mu_{K,\sigma}$ converges to the uniform measure on $X(1) \times X(1)$ whenever the discriminant of $K$ and the minimal norm of an 
invertible ideal with ideal class $\sigma$ both tend to $\infty$. 

The work of \citet{TS} is a function field analogue of this mixing conjecture. The analogy is constructed by replacing $\mathbb Q$ with $\mathbb F_q(T)$, $X(1)$ with the set $\operatorname{Bun}_2(\mathbb P^1)$, $K$ with the function field of $C$ over $\mathbb F_q$, $Cl(K)$ with $\operatorname{Pic}(C) / \operatorname{Pic}( \mathbb P^1)$, and $z_{K,\alpha}$ with $\pi_* L$. In this setting, Theorem \ref{equidistribution-statement} is exactly the analogue of the conjecture of Michel and Venkatesh (once the trivial but necessary determinant mod $2$ condition is dealt with).

The cohomological conjecture \citep[Conjecture 1.4]{TS} needed to prove this mixing result was proven in characteristic zero by \citet[Theorem 1.5]{TS}, using Massey's bounds for the stalks of perverse sheaves. Thus it was natural to approach the conjecture in characteristic $p$ using Theorem \ref{first-Massey-intro}. Our arguments to prove Theorem \ref{equidistribution-statement} follows closely the proof of \citep[Theorem 1.5]{TS}. One modification needed is that, in characteristic zero, one knows that all irreducible components of the characteristic cycle are Lagrangian varieties, hence are the conormal bundle to their supports, and one can thus calculate the characteristic cycle by studying only these supports. In characteristic $p$, these irreducible components are not necessarily Lagrangian, and so it is necessary to perform calculations in the cotangent bundle, not on the base variety. In addition, the argument uses a new idea provided by Tsimerman in the appendix to bound a crucial multiplicity.

Since the writing of \citep{TS}, the equidistribution conjecture on $X(1) \times X(1)$ was verified by \citet[Theorem 1.3]{Khayutin}, using ergodic theory methods. In addition, \citet{Khayutin} proved this statement over modular curves of higher level (while Theorem \ref{equidistribution-statement} requires level $1$.) However, this required two additional assumptions: that the fields $K$ are always split at two fixed primes $p_1, p_2$, and that their Dedekind zeta functions have no Landu-Siegel zero.  

In comparing these results, one should note that (unlike some results over $\mathbb Q$) it is not yet clear if the argument of \citet{Khayutin} can be made to work over function fields, as there are more measures to rule out. See \citep*[Theorem 1.2 and \S1.3]{ELM} for a measure classification result and a discussion of the difficulties arising from measures invariant under subgroups defined over subfields, of which $\mathbb F_q(T)$ has infinitely many. Such a transfer would allow one to remove the level 1 assumption from Theorem \ref{equidistribution-statement}, at the cost of introducing the split primes assumption. Going from the function field to the number field case, on the other hand, is as hard as usual.

\subsection{Acknowledgments}

The author was supported by Dr. Max R\"{o}ssler, the Walter Haefner Foundation and the ETH Zurich Foundation, and, later, served as a Clay Research Fellow, while working on this research. I would like to thank Takeshi Saito, Vivek Shende, and Jacob Tsimerman for helpful discussions about their works, Philippe Michel and Manfred Einsiedler for helpful comments about the general equidistribution problem, and the anonymous referee for helpful comments.

\section{Terminology}

We review some notation and terminology from \cite{Beilinson} and \cite{saito1}. (Our formulations of the definitions are mainly adapted from \cite{saito1}). All schemes are over a perfect field $k$, which in the application we can specialize to be the algebraic closure of a finite field.

\begin{definition}\label{C-transversal-1} \cite[Definition 3.5(1)]{saito1} Let $X$ be a smooth scheme over $k$ and let $C \subseteq T^* X$ be a closed conical subset of the cotangent bundle. Let $f: X\to Y$ be a morphism of smooth schemes over $k$.

We say that $f: X \to Y$ is \emph{$C$-transversal} if the inverse image $df^{-1}(C)$ by the canonical morphism $X \times_Y T^* Y \to T^* X$ is a subset of the zero-section $X \subseteq X \times_Y T^* Y$.
\end{definition}

\begin{definition}\cite[(1.2)]{Beilinson}\label{circ-forward} In the same setting as Definition \ref{C-transversal-1}, if $f$ is proper, let $f_\circ C$ be pushforward from $X \times_Y T^* Y$ to $T^*Y $ of the the inverse image $df^{-1}(C)$.

\end{definition}

\begin{definition}\cite[(2.3)]{saito-direct}\label{shriek-forward}  In the same setting as Definitions \ref{C-transversal-1} and \ref{circ-forward}, let $A$ be an algebraic cycle of codimension $\dim X$ supported on $C$. Assume also that $f_\circ C$ has dimension $\dim Y$.

Let $f_!  A$ be the pushforward from $X \times_Y T^* Y$ to $T^*Y $ of the intersection-theoretic inverse image $df^* C$. \end{definition}

\begin{definition}\cite[Definition 3.1]{saito1} Let $X$ be a smooth scheme over $k$ and let $C \subseteq T^* X$ be a closed conical subset of the cotangent bundle. Let $h: W \to X$ be a morphism of smooth schemes over $k$.

Let $h^* C$ be the pullback of $C$ from $T^* X$ to $W \times_X T^* X$ and let $K$ be the inverse image of the $0$-section $W \subseteq T^* W$ by the canonical morphism $dh: W \times_X T^* X \to T^* W$. 

We say that $h: W\to X$ is \emph{$C$-transversal} if the intersection $h^* C \cap K$ is a subset of the zero-section $W \subseteq W \times_X T^* X$.

If $h: W \to X$ is $C$-transversal, we define a closed conical subset $h^\circ C \subseteq T^* W$ as the image of $h^* C$ under $dh$ (it is closed by \cite[Lemma 3.1]{saito1}). \end{definition}

\begin{definition}\cite[Definition 3.5(2)]{saito1} We say that a pair of morphisms $h: W \to X$ and $f:W\to Y$ of smooth schemes over $k$ is \emph{$C$-transversal}, for $C \subseteq T^* X$ a closed conical subset of the cotangent bundle, if $h$ is $C$-transversal and $f$ is $h^\circ C$-transversal. \end{definition}

\begin{definition}\cite[Th. Finitude, Definition 2.12]{sga4h} We say that a morphism $f: W \to Y$ is locally acyclic relative to $K \in D^b_x(W, \mathbb F_\ell)$ if for each geometric point $x \in W$ geometric point  $t \in Y$ specializing to $f(x)$, $W_x$ the henselization of $W$ at $x$ and $W_{x,t}$ the fiber of $X_x$ over $t$, the natural map $H^* ( W_x, K) \to H^* ( W_{x,t}, K)$ is an isomorphism.  \end{definition}

For us, the main advantage of local acyclicity is that, when $Y$ is a smooth curve, the local acyclicity of $f$ implies that $R\Phi_f K$ vanishes, since taking $t$ the generic point, $H^* ( W_{x,t}, K)$ is the stalk of $R \Psi_ K$ at $x$ and $H^* ( W_x, K) $ is the stalk of $K$ at $x$, so the map is an isomorphism if and only if the mapping cone $R(\Phi_f K)_x$ vanishes.

\begin{definition}\cite[1.3]{Beilinson} For $K \in D^b_c(X, \mathbb F_\ell)$, let the \emph{singular support $SS(K)$ of $K$} be the smallest closed conical subset $C \in T^* X$ such that for every $C$-transversal pair $h: W \to X$ and $f: W\to Y$, the morphism $f: W\to Y$ is locally acyclic relative to $h^* K$. 

\end{definition}
The existence and uniqueness of $SS(K)$ is \cite[Theorem 1.3]{Beilinson}, which also proves that if $X$ has dimension $n$ then $SS(K)$ has dimension $n$ as well.

\begin{definition}\cite[Definition 7.1(1)]{saito1} Let $X$ be a smooth scheme of dimension $n$ over $k$ and let $C \subseteq T^* X$ be a closed conical subset of the cotangent bundle with each irreducible component of dimension $n$. Let $W$ be a smooth scheme of dimension $m$ over $k$ and let $h: W \to X$ be a morphism over $k$.

We say that $h$ is \emph{properly $C$-transversal} if it is $C$-transversal and each irreducible component of $h^* C$ has dimension $m$. \end{definition} 

\begin{definition}\cite[Definition 7.1(2)]{saito1} Let $X$ be a smooth scheme of dimension $n$ over $k$ and let $A$ be an algebraic cycle of codimension $n$ on $\subseteq T^* X$ whose support $C$ is a closed conical subset of the cotangent bundle (necessarily of dimension $n$).

 Let $W$ be a smooth scheme of dimension $m$ over $k$ and let $h: W \to X$ be a properly $C$-transversal morphism over $k$.
 
 We say that $h^{!} A$ is $(-1)^{n-m}$ times the pushforward along $dh:  W \times_X T^* X \to T^* W$ of the pullback along $h:  W \times_X T^* X \to T^* X$ of $A$, with the pullback and pushforward in the sense of intersection theory.  \end{definition}

Here the pushforward in the sense of intersection theory is well-defined because, by \cite[Lemma 3.1]{saito1}, $dh$ is finite when restricted to (the induced reduced subscheme structure) on $h^* C$, i.e finite when restricted to the support of $h^* A$.

\begin{definition} \cite[Definition 5.3(1)]{saito1} Let $X$ be a smooth scheme of dimension $n$ over $k$  and let $C\subseteq T^* X$ be a closed conical subset of the cotangent bundle. Let $Y$ be a smooth curve over $k$ and $f: X\to Y$ a morphism over $k$. 

We say a closed point $x \in X$ is at most an \emph{isolated $C$-characteristic point} of $f$ if $f$ is $C$-transversal when restricted to some open neighborhood of $x$ in $X$, minus $x$.  We say that $x\in X$ is an \emph{isolated $C$-characteristic point} of $f$ if this holds, but $f$ is not $C$-transversal when restricted to any open neighborhood of $X$. \end{definition}

\begin{definition} For $V$ a representation of the Galois group of a local field over $\mathbb F_\ell$ (or a continuous $\ell$-adic representation), we define $\operatorname{dimtot} V$ to be the dimension of $V$ plus the Swan conductor of $V$. For a complex $W$ of such representations, we define $\operatorname{dimtot}W $ to be the alternating sum $\sum_i (-1)^i \operatorname{dimtot} \mathcal H^i(W)$ of the total dimensions of its cohomology objects. \end{definition}

\begin{definition}\cite[Definition 5.10]{saito1} Let $X$ be a smooth scheme of dimension $n$ over $k$ and $K$ an object of $D^b_c(X, \mathbb F_\ell)$. Let the \emph{characteristic cycle of $K$}, $CC(K)$ , be the unique $\mathbb Z$-linear combination of irreducible components of $SS(K)$ such that for every \'{e}tale morphism $j: W \to X$, every morphism $f: W\to Y$ to a smooth curve and every at most isolated $h^\circ SS(\mathcal F)$-characteristic point $u \in W$ of $f$, we have

\[ - \operatorname{dimtot} \left( R \Phi_f(j^* K) \right)_u =  (j^* CC(K), (df)^*\omega )_{T^*W,u} \] where $\omega$ is a meromorphic one-form on $Y$ with no zero or pole at $f(u)$.

\end{definition}

Here the notation $(,)_{T*W ,u}$ denotes the intersection number in $T^* W$ at the point $u$.

The existence and uniqueness is \cite[Theorem 5.9]{saito1}, except for the fact that the coefficients lie in $\mathbb Z$ and not $\mathbb Z[1/p]$, which is \cite[Theorem 5.18]{saito1} and is due to Beilinson, based on a suggestion by Deligne.

\section{Equivalences between definitions of the polar multiplicity}\label{equivalences}

In this section we give an alternate definition of the polar multiplicity, check that it is equivalent to the previous one, and check that both are well-defined.

\begin{defi} Let $Y$ be a smooth variety with a map $f$ to a variety $X$ (which may be the identity), and let $x$ be a point on $X$. Let $C_1, C_2$ be algebraic cycles on $Y$ of total dimension $\dim Y$ such that $C_1 \cap C_2 \cap f^{-1}(x)$ is proper. Assume that all connected components of $C_1 \cap C_2$ are either contained in $f^{-1}(x)$ and proper or disjoint from $X$. We define their intersection number locally at $x$  \[ (C_1,C_2)_{Y, x} \] to be the sum of the degrees of the refined intersection $C_1 \cdot C_2$ \cite[p. 131]{Fulton} on all connected components of $C_1 \cap C_2$ contained in $f^{-1}(x)$. \end{defi}

\begin{lemma}\label{intersection-comparison} Let $X$ be a smooth variety. Let $C$ be a conical cycle on the cotangent bundle of $X$ of dimension $\dim X$ and let $x$ be a point on $X$.  Let $\mathbb P(C) \subseteq \mathbb P (T^* X)$ be the projectivization of $C$ inside the projectivization of the cotangent bundle of $X$. Let $i$ be a natural number with $0 \leq i < \dim X$.

Consider $Y \subset X$ a smooth variety of dimension $\dim X-i$ through $x$ and $V$ a sub-bundle of $T^* X$  of rank $i+1$ on $Y$.  Let $\mathbb P(V) \subseteq \mathbb P( T^* X)$ be the projectivization of $V$ over $Y$. For any $(Y,V)$ such that the strict transforms of $\mathbb P(C)$ and $\mathbb P(V)$ in the blowup of $\mathbb P(T^* X)$ at the fiber over $X$ do not intersect inside the exceptional divisor, the contribution of the fiber over $x$ to the intersection number $\mathbb P(C) \cap \mathbb P(V)$ depends only on $i$ and is independent of $Y,V$.

Furthermore, to satisfy the condition on the strict transform, it is sufficient that the tangent space of $Y$ at $x$ and the fiber of $V$ over $x$ are independent generic subspaces of the tangent and cotangent spaces of $X$ at $x$ respectively. In particular, such a $Y$ and $V$ exist.\end{lemma}

\begin{proof} This is a local question, and we may work locally.  Then given $(Y,V)$ and $(Y',V')$ both satisfying this condition, we may deform one into the other by a connected family of varieties. For instance we may represent $Y$ and $Y'$ as local complete intersections and deform the equations defining $Y$ into the equations defining $Y'$ by convex combination, and similarly for the vector subbundles defining $Y'$ and $V'$. The condition that the intersection of the strict transforms vanishes is an open condition, because the strict transform of $\mathbb P(V)$ varies properly with $Y$ and $V$, so we may assume that there is a family connecting $(Y,V)$ to $(Y',V')$ where every member satisfies this condition. Then because the intersection locus in the blow-up is closed, its image inside $X$ is too, and because it is disjoint from $x$, there must be some neighborhood of $X$ that it doesn't intersect. Then for any $Y_t,V_t$ in the family, the intersection of $\mathbb P(C)$ and $\mathbb P(V_t)$ in $\mathbb P (T^* X)$ is empty in that neighborhood minus $x$, so the contribution to the intersection coming from the fiber over $x$ is constant in the family, and thus is equal for $(Y,V)$ and $(Y',V')$.

For the claim about generic subspaces, note that $C$ has dimension $\dim X$, so $\mathbb P(C)$ has dimension $\dim X-1$, and the intersection of its strict transform with the fiber has dimension $\dim X-1$.  The fiber of the blowup is isomorphic to $\mathbb P( (TX)_x ) \times \mathbb P( (T^*X)_x )$, of dimension $2 \dim X - 2$, and the strict transform of $\mathbb P(V)$ is $\mathbb P((TY)_x) \times \mathbb P ( V_x)$, of dimension $\dim X - i -1 + i = \dim X-1$. If we take $(TY)_x$ and $V_x$ to be general subspaces, this intersection will have the expected dimension, which is $-1$, and hence be empty. \end{proof}

\begin{defi}\label{polar-multiplicity-2} Let $X$ be a smooth variety. Let $C$ be a conical cycle on the cotangent bundle $T^* X$ of $X$ of dimension $\dim X$ and let $x$ be a point on $X$.

For $0 \leq i< \dim X$, let $Y$ be a sufficiently general smooth subvariety of $X$ of codimension $i$ passing through $x$ and let $V$ be a sufficiently general sub-bundle of $T^* X$ over $Y$ with rank $i+1$. Define the $i$th polar multiplicity of $C$ at $x$ to be the intersection number \[ ( \mathbb P(C) , \mathbb P(V))_{ \mathbb P(T^* X), x} \] where $\mathbb P(T^* X)$ is the projectivization of the vector bundle $T^* X$.

Here ``sufficiently general" means that the strict transform of $\mathbb P(V)$ in the blowup of $\mathbb P(T^* X)$ at the fiber over $x$ does not intersect the strict transform of $\mathbb P(C)$ in that same blowup within the fiber over $x$.

For $i=\dim X$, define the $i$th polar multiplicity of $C$ at $x$ to be the multiplicity of the zero section in $C$. \end{defi}

It follows from Lemma \ref{intersection-comparison} that this is well-defined.

\begin{lemma}\label{multiplicity-polar} Definitions \ref{polar-multiplicity} and \ref{polar-multiplicity-2} are equivalent.
 \end{lemma}

\begin{proof} By definition, the multiplicity of an algebraic cycle at a point is the local intersection number with a sufficiently general smooth scheme passing through that point. For $Y$ a sufficiently general smooth subscheme of $X$ of dimension $n-i$, we have an identity of intersection numbers \[ (\pi_* ( \mathbb P(C) \cap \mathbb P(V) , Y)_{X,x}) = (\mathbb P(C) \cap \mathbb P(V), \pi^* Y)_{\mathbb P(T^*X), X} =  (\mathbb P(C), \mathbb P(V) \cap \pi^* Y)_{\mathbb P(T^*X), X} .\]

Note that $\mathbb P(V) \cap \pi^* Y$ is simply the projectivization of the restriction $V'$ of $V$ to $Y$. So to check that this is the polar multiplicity, it suffices to check that if $V_x$ is sufficiently general, and $Y$ is sufficiently general depending on $V$, that the restriction of $V$ to $Y$ is sufficiently general in the sense of Definition \ref{polar-multiplicity-2}. This occurs when the intersection of the strict transform of $\mathbb P(C)$ with the strict transform of $\mathbb P(V')$ in the exceptional divisor of the blowup of $\mathbb P (T^* X)$ at the fiber over $x$ vanishes.

The exceptional divisor is isomorphic to $\mathbb P ( (TX)_x ) \times \mathbb P( (T^*X)_x) $. Inside it, the strict transform of $\mathbb P(V')$ is $\mathbb P ( (TY)_x) \times \mathbb P (  V_x)$. The intersection of the strict transform of $\mathbb P(C)$ with the exceptional divisor has dimension at most $\dim \mathbb P(C) -1 =\dim C-2= n-2$. For $V$ of dimension $i+1$, $\mathbb P( (TX)_x) \times \mathbb P( V_x)$ has codimension $n-i-1$, so for $V_x$ sufficiently general,  the intersection of the strict transform with $( \mathbb P( (TX)_x) \times \mathbb P( V_x))$ has dimension $i-1$. Then for general $\mathbb P( (TY)_x)$ of codimension $i$, the intersection of $\mathbb P ( (TY)_x) \times \mathbb P (  V_x)$ with the strict transform is empty.  
\end{proof}

\section{A bound for Betti numbers}

\begin{lemma}\label{transversality-check} Let $f: X \to Y$ be a smooth morphism of smooth varieties with $X$ of dimension $n$ and $Y$ of dimension $n-m$. Let $C$ be a closed conical subset of the cotangent bundle $T^* X$ of $X$ with all irreducible components of dimension  $n$. Let $y$ be a point in $Y$ and let $i$ be the inclusion of $f^{-1}(y)$ into $X$, so that we have a Cartesian square. \begin{equation}\label{i-diagram}\begin{tikzcd} f^{-1}(y) \arrow[r,"i"] \arrow[d] & X \arrow[d,"f"] \\ y \arrow[r] & Y \end{tikzcd}\end{equation}

 If $f$ is $C$-transversal and the fibers of the composition $C \to X \to Y$ have dimension $m$, then $i$ is properly $C$-transversal. \end{lemma}

\begin{proof} Because $f$ is $C$-transversal, the inverse image of $C$ by $df: X \times_Y T^* Y \to T^* X$ is a subset of the zero-section. Hence the intersection of $C$ with the image of $df$ is a subset of the zero-section, as only nonzero points are sent to nonzero points by $df$. The image of $df$ in $T^* X$ consists of $1$-forms that are pulled back from $Y$, i.e. one-forms that are transverse to the fibers of $X$, which are exactly those one-forms in the kernel of $di: f^{-1}(y) \times_X T^* X \to T^* f^{-1}(y)$. Hence $i$ is $C$-transversal.

$i^* C = f^{-1}(y) \times_X C =  y \times_Y C$ is exactly a fiber of the composition $C \to Y$, and thus the claim that it has dimension $\dim X - \dim Y = \dim (f^{-1}(y))$ verifies that $i$ is properly $C$-transversal.

\end{proof}

\begin{lemma}\label{nearby-formula} Let $X$ be a smooth variety and $Y$ a smooth curve, both over a perfect field $k$. Let $K$ be an object in $D^b_c(X,\mathbb F_\ell)$.

Let $CC'(K)$ be $CC(K)$ with any occurrence of the cotangent space at $x$ removed, and let $SS'(K)$ be $SS(K)$ with any occurrence of the cotangent space at $x$ removed. 

Let  $f: X \to Y$ be a smooth projective morphism that is $SS'(K)$-transversal and such that the fibers of $SS'(K)$ over $Y$ have dimension $\dim X-1$ . Let $y=f(x)$, let $Z =f^{-1}(y)$ and let $i$ be the inclusion of $Z$ into $X$, as in Diagram \ref{i-diagram}.  Then 
\[i^! CC'(K) = CC ( R \Psi_{f} K) \]
where we view the nearby cycles relative to $f$ as a complex of sheaves on $f^{-1}(Y)$. \end{lemma}

Note that $SS(K)$ here is a union of irreducible varieties of dimension $\dim X$ and $CC(K)$ is a $\mathbb Z$-linear combination of irreducible varieties of dimension $\dim X$. When we refer to removing the cotangent space at $X$, an irreducible variety of dimension $\dim X$, we mean removing this term from the $\mathbb Z$-linear combination or the union, if it appears, but leaving all other terms.

\begin{proof} Because $f$ is $SS'(K)$-transversal, it is $SS(K)$-transversal away from $x$, so $K$ is locally acyclic away from $f$ by the definition of the singular support, and thus $R \Phi_f K$ vanishes away from $x$, so $R \Psi_f K = i^* K$ away from $x$.

Furthermore, $i$ is properly $SS'(K)$-transversal by Lemma \ref{nearby-formula}.

Then by \cite[Theorem 7.6]{saito1} \[  CC ( R \Psi_{f} K) = CC( i^* K ) = i^! CC(K) = i^! CC'(K) \] away from $x$. 

Hence \[i^! CC'(K) -  CC ( R \Psi_{f} K) \] is a cycle on the cotangent bundle of $Z$ supported inside the cotangent space at $x$. Because these cycles are rational linear combinations of irreducible closed sets of dimension $\dim Z$, the difference is a multiple of the cotangent space at $x$. Because the cotangent space at $x$ has nonzero intersection number with the zero-section $Z$ of $T^* Z$, to check that \[ i^! CC'(K) = CC(R\Psi_f K),\] it suffices to check \[ (i^! CC'(K) , Z)_{T^*Z}=   (CC ( R \Psi_{f} K) , Z)_{T^*Z} . \] By the index formula \cite[Theorem 7 .13]{saito1},  \[ (CC ( R \Psi_{f} K) , Z)_{T^*Z} =  \chi(Z, R \Psi_f K) =\chi ( f^{-1}(\eta),K)\] for $\eta$ the generic point of $Y$. 

By definition, \[i^! CC'(K) = - (di)_* i^*  CC'(K).\] We have \cite[Proposition 8.1.1(c)]{Fulton}
\[  ((di)_* i^* CC'(K), Z)_{T^*Z} = (i^* CC'(K),  (di)^* Z)_{T^*X \times_X Z} = (CC'(K), i_* (di)^* Z) _{T^*X} .\]
For the first identity, this uses the fact that $di$ is finite on the support of $i^* CC'(K)$ and for the second identity this uses the fact that $i$ is a closed immersion, hence finite.

Now $(di)^ * Z \subseteq  T^* X \times_X Z$ consists of one-forms transverse to $Z$, so $i_* (di)^* Z$ is the conormal bundle of $Z$ inside $X$, $N^*Z$. 

As a point $y' \in Y$ varies, the conormal bundle to $f^{-1}(y')$ varies in a smooth family. To check that the intersection number  \[ (CC'(K), N^*f^{-1} (y') )_{T^*X} \] is constant, it suffices to check that the the intersection $ CC'(K) \cap N^*f^{-1} (y') $, viewed as a family of closed subsets parameterized by $y' \in Y$, and hence viewed as a scheme mapping to $Y$ given the induced reduced subscheme structure, is proper over $Y$. This is true because, as $f$ is $SS'(K)$-transversal, this intersection is contained in the zero-section, hence is proper.

The same is true for any other $y'$, and these conormal bundles vary in a smooth family, so this intersection number for $y$ is equal to the intersection number for any $y'$, and in particular for the generic point $\eta$. Let $i_\eta$ be the inclusion of the generic fiber of $f$ into $X$, then \[( i^! CC'(K), Z)_{T^*Z}  = - (CC'(K),  N^* f^{-1}(y))_{T^*X} = (CC'(K), N^* f^{-1}(\eta))_{T^*X}\] 
\[ = (i_\eta^! CC'(K), f^{-1}(\eta))_{T^* f^{-1}(\eta)}  =  (i_\eta^! CC(K), f^{-1}(\eta))_{T^* f^{-1}(\eta)} =  ( CC(i_\eta^* K) , f^{-1}(\eta))_{T^* f^{-1}(\eta)}\] \[= \chi(f^{-1}(\eta),K) \] as desired, where the first equality summarizes the previous calculations.\end{proof}

\begin{lemma}\label{genericity-lemma} Let $X$ be a smooth variety embedded in projective space $\mathbb P^n$. Let $C$ be a closed conical subset of the cotangent space of $X$ of dimension $\dim X$. Let $x \in X$ be a point such that $C$ does not contain the cotangent space of $x$.

Let $\overline{X} \subseteq X \times \mathbb P^1$ be a general pencil of conic sections of $X$, parameterized by $\mathbb P^1$.  Let $p$ and $q$ be the natural projections in the below commutative diagram.\begin{equation}\label{pq-diagram}\begin{tikzcd} \overline{X}  \arrow[ddr,"p"'] \arrow[drr,"q"] \arrow[dr]  \\  &X \times \mathbb P^1 \arrow[r] \arrow[d] & \mathbb P^1 \\ & X \end{tikzcd} \end{equation}

Then $p$ is properly $C$-transversal, $q$ is $p^{\circ} C$-transversal in a neighborhood of the unique conic in the pencil containing $x$, and the fibers of $p^\circ C$ over $\mathbb P^1$ are $\dim X-1$-dimensional in a neighborhood of the conic containing $x$. \end{lemma}

\begin{proof}

Let $Y$ be the base locus of this pencil of conics, which since the pencil is generic, is a smooth subscheme of codimension $2$ (by Bertini's theorem). We can view $p$ as the blow-up of $X$ along $Y$.

First we check that $p$ is $C$-transversal. The map $p$  is \'{e}tale, and automatically $C$-transversal, away from $Y$, and at each point over $Y$, $dp^{-1} ( \{0\})$ is one-dimensional and contained in the conormal space of $Y$. Thus to check that $p$ is $C$-transversal, it suffices to check that no point in $C$ consists of a point in $Y$ and a nonzero vector transverse to $Y$.  For each pair of a point and nonzero cotangent vector in $C$, the condition that the point be contained in $Y$ is a codimension $2$ condition on the pencil of conics, and the condition that the vector be transverse to $Y$ is a codimension $\dim X-2$ condition on the pencil of conics. Because the space of pairs of a point and a nonzero cotangent vector contained in $C$, up to dilation of the cotangent vector, is $\dim X-1$-dimensional, this occurring for any point is a codimension $1$ condition, hence is not generic.

Next we check that $p$ is properly $C$-transversal.  Because $\dim \overline{X} =\dim X$ and the fibers of $p$ have dimension at most one, $\dim p^* C = \dim C = \dim X = \dim \overline{X}$ unless the base of some irreducible component of $C$ lies entirely in $Y$. For any given variety, a generic $Y$ does not contain it, so this does not happen, and $p$ is properly $C$-transversal.

 Let $t$ in $\mathbb P^1$ be such that $x \in \overline{X}_t$. Now we check that $q$ is $p^\circ C$-transversal in a neighborhood of $X_t$.  Because $dq^{-1}  (p^\circ C)$ is a closed conical subset of $ \overline{X} \times_{\mathbb P^1} T^* \mathbb P^1$, and a closed conical subset being contained in the zero section is an open condition, $q$ being $p^\circ C$-transversal is an open condition.
 
 Thus it suffices to check that the restriction of $dq^{-1} (p^\circ C)$ to $X_t$ is contained in the zero section. Equivalently, we fix a nonzero one-form $\omega_0$ on $\mathbb P^1$ at $t$, so that $dq(\omega_0)$ generates the one-dimensional image of $dq$, and check that $dq(\omega_0)_z \notin p^\circ C$ at each point $y\in X_t$. 
 

If $y\in Y$, the image of $dp$ and $dq$ intersect only at zero, so $dq(\omega_0)_z \notin p^\circ C$. If $z\notin Y$,  $dq(\omega_0 )$ is the conormal vector to  $X_t$, because $X_t$ is a level set of $q$. Thus, to check that  $q$ is generically $p^\circ C$-transversal in a neighborhood of $\overline{X}_t$, it suffices to check that, for a general conic $\overline{X}_t$ through $x$, the conormal bundle to $\overline{X}_t$ never contains a pair of a point and a nonzero cotangent vector in $C$. 

For each point $y$ and nonzero cotangent vector $\omega_1$ in $C$, with $y\neq x$, the conics through $x$ whose conormal bundles contain $(y, \omega_1)$ form a codimension $\dim X$ subset of the conics through $x$, because this is a codimension one condition on the value of the conic at $y$ and a codimension $\dim X-1$ condition on the derivative of the conic at $y$, and the derivatives at $y$ are independent of the condition that the conic pass through $x$. Because the space of pairs of a point $y$ and a nonzero cotangent vector $\omega_1$ contained in $C$, up to dilation of $\omega_1$, is $\dim X-1$-dimensional, this is a codimension $1$ condition and is not generic. Over the point $x$, the conormal bundle of.a generic conic is a general cotangent line, so it remains to check that $C$ does not contain a general point of the cotangent space at $x$, which holds because we have assumed that the cotangent space at $x$ does not lie in $C$.

Because each irreducible component of $p^\circ C$ has dimension $\dim X$, the fibers over $\mathbb P^1$ have dimension $\dim X -1$ unless some irreducible component is contained entirely in one fiber, i.e. in a single conic in the pencil.  For a generic pencil of conics, the only variety that is necessarily contained in one fiber of the pencil is a single point, and because $C$ does not contain the cotangent space of $x$, none of these points are $x$, and so they will not generically be in the same fiber as $x$, and thus we can remove the fibers containing these points from our chosen neighborhood.
\end{proof}

\begin{lemma}\label{final-nearby-formula} Let $X$ be a smooth variety embedded in projective space $\mathbb P^n$ over a perfect field $k$. Let $K$ be an object of $D^b_c(X, \mathbb F_\ell)$.

Let $CC'(K)$ be $CC(K)$ with any occurrence of the cotangent space at $x$ removed, and let $SS'(K)$ be $SS(K)$ with any occurrence of the cotangent space at $x$ removed.

Let $\overline{X} \subseteq X \times \mathbb P^1$ be a general pencil of conic sections sections of $X$, parameterized by $\mathbb P^1$.  Let $p$ and $q$ be the projections $\overline{X} \to X$ and $\overline{X} \to \mathbb P^1$, respectively, as in Diagram \ref{pq-diagram}. 

Let $\overline{x} = p^{-1}(x)$, which, because the pencil is generic, is a single point, and let $y= q(\overline{x})$. Let $i : q^{-1}(y) \to \overline{X}$ be the inclusion of the fiber over $y$ into $\overline{X}$, so that we have the commutative diagram. \begin{equation}\label{pqi-diagram}\begin{tikzcd} \overline{X}  \arrow[ddr,"p"'] \arrow[drr,"q"] \arrow[dr]  &  &  q^{-1}(y) \arrow[ll, "i"'] \arrow[dr] \\  &X \times \mathbb P^1 \arrow[r] \arrow[d] & \mathbb P^1 & y \arrow[l]  \\ & X \end{tikzcd} \end{equation}

Then

\begin{enumerate}

\item \[ CC ( R \Psi_{q} p^*K) = i^! p^! CC'(K)\]

\item  $R\Phi_q p^* K$ is supported at $x$.

\item  \[-\operatorname{dimtot}  \left( R\Phi_q ( p^* K) \right)_x  \] is the multiplicity of the cotangent space at $x$ in $CC(K)$.

\item If $K$ is perverse, then  $\left( R\Phi_q ( p^* K) \right)_x $ is supported in degree $-1$.

\end{enumerate}
  \end{lemma}

\begin{proof} To obtain (1), we apply Lemma \ref{nearby-formula} to $p^* K$ and $q$. By Lemma \ref{genericity-lemma}, $p$ is properly $SS'(K)$-transversal, and it is \'{e}tale at $x$ so it is properly $SS(K)$-transversal, so by \cite[Theorem 6.6]{saito1}, \[CC(p^* K) = p^!(CC(K))\] and thus \[CC'(p^* K) = p^!(CC'(K)).\] Then by Lemma \ref{genericity-lemma}, $p^* K$ satisifies the conditions of Lemma \ref{nearby-formula}, so \[ CC ( R \Psi_{q} p^*K) = i^! CC'(p^* K) = i^! p^! CC'(K) \] as desired. 

To obtain (2), by Lemma \ref{genericity-lemma}, $q $ is $ p^\circ SS(K)= SS(p^* K)$-transversal in a neighborhood of $x$, minus $x$, hence  $p^* K$ is locally $q$-acyclic away from $x$  by the definition of the singular support,  and thus $R \Phi_q p^* K$ is supported at $x$.

For (3), to calculate $R \Phi_q p^*K $, we use again the fact that $q $ is $ SS'(p^* K)$-transversal, so it is $SS(p^*K)$-transversal away from $x$, and thus $x$ is at most an isolated characteristic point, so by the definition of the characteristic cycle \[ -\operatorname{dimtot} (R \Phi_q p^* K)_x = (CC (p^* K), (dq)^* \omega)_x\] where $\omega$ is nonvanishing one-form on an open neighborhood of $y$ in $\mathbb P^1$.

Because $q$ is $SS'(p^*K)$-transversal, the only irreducible component of $SS(p^*K)$ which intersects $(dq)^* \omega $ is the cotangent space $N^* X$ at $x$. Because $(dq)^* \omega$ is a section of the cotangent bundle, $(N^*x, (dq)^* \omega)= 1$, so $(SS(p^*K), (dq)^* \omega)$ is the multiplicity of $N^*x$ in $CC(p^K)$, which is also the multiplicity of $N^* x$ in $CC(K)$.

For (4), because $p^*K$ is perverse near $x$, $p^* K[-1]$ is perverse near $x$ when restricted to the generic fiber of $q$, and so by the theorem of Gabber \cite[Corollary 4.6]{Illusie}, $R\Phi_q ( p^* K)[-1]$ is perverse (near $x$, and thus everywhere, because it vanishes elsewhere). Because it is perverse and supported at a single point, it is supported in degree $0$, and then the unshifted version is supported in degree $[-1]$.

\end{proof}

\begin{lemma}\label{multiplicity-hypersurface}  Let $X$ be a smooth variety embedded in projective space $\mathbb P^n$ over a perfect field $k$. Let $x$ be a point of $X$. Let $C$ be a conical cycle in the cotangent bundle of $X$.  Let $C'$ be $C$ minus any occurrence of the cotangent space at $x$.  Let $\inhyper{X}$ be the intersection of $X$ with a generic conic through $x$, let $j: \tilde{X} \to X$ be the inclusion, and let $\tilde{C} =-  j^{!} C'$.

Then for $i>0$, the $i$th polar multiplicity of $C$ at $x$ equals the $i-1$st polar multiplicity of $\inhyper{C}$ at $x$, and for $i=0$, the $i$th polar multiplicity of $C$ at $X$ equals the multiplicity of the cotangent space at $x$ in $C$. \end{lemma}

\begin{proof} We split into three cases: $i=0$, $0< i < \dim X$, $i=\dim X$.

For $i=0$, in the definition of polar multiplicity we can let $Y =X$, with $V$ a rank one subbundle of the cotangent bundle, so $\mathbb P(V)$ is simply a section of $\mathbb P(T^* X)$. If we choose a general section, then the only irreducible component of $\mathbb P(C)$ it intersects at $x$ is the fiber over $x$, which it intersects with multiplicity the multiplicity of that fiber, which is the multiplicity of the cotangent space at $x$ in $C$.

For $0< i < \dim X$, let $Y$ be a general smooth subvariety of $\inhyper{X}$ of codimension $i-1$ passing through $x$ and let $V$ be a general $i$-dimensional sub-bundle of $T^* \inhyper{X}$ over $Y$. Let $V$ be the inverse image of $\inhyper{V}$ in the cotangent bundle of $X$. By Definition \ref{polar-multiplicity-2}, it suffices to check that \[( \mathbb P( \inhyper{C}), \mathbb P(\inhyper{V}))_{\mathbb P( T^* \inhyper{X}),x}  =(  \mathbb P(V) , \mathbb P(C) )_{\mathbb P(T^* X),x}\] and that $Y,V$ satisfies the conditions of Lemma \ref{intersection-comparison} for $X,C$.

Let $s \colon \inhyper{X} \to  \mathbb P(T^* X \times_X \inhyper{X})$ be the section of $\mathbb P(T^* X \times_X \inhyper{X})$ corresponding to the conormal line of $\inhyper{X}$. After replacing $X$ by a suitable neighborhood of $x$, we have a commutative diagram.
\begin{equation}  \begin{tikzcd}  \mathbb P( C' ) \times_X \tilde{X}  \arrow[d,"z" ]  \arrow[r, "t" ] & \mathbb P ( T^* \tilde{X})  \arrow[r,"e"] & \tilde{X} \\
  \mathbb P(T^*X \times_X \inhyper{X}  ) - s(\inhyper{X}) \arrow[r,"u" ] & \operatorname{Bl}_{ s( \inhyper{X} ) } \mathbb P(T^*X \times_X \inhyper{X}  ) \arrow[u,"\rho"] \arrow[r, "b"] &  \mathbb P(T^*X \times_X \inhyper{X}  )  \arrow[u] & \tilde{X} \arrow[l,"s" ] \arrow[ul] \end{tikzcd} \end{equation}
In this diagram, the maps $u, \rho, b$ arise from projective geometry - the blow up of a projective bundle at a section $s$ admits a projection map $\rho$ to the projectivization of the quotient vector bundle, and a map $u$ from the open complement of this section. The inclusion $z\colon   \mathbb P( C' ) \times_X \tilde{X}  \to \mathbb P(T^*X \times_X \inhyper{X}  ) - s(\inhyper{X}) $ exists because $s(x)$ is a general point of $\mathbb P( (T^*X)_x)$ and $C'$ does not contain the whole cotangent space of $X$ at $x$, so $s(x)$ is not contained in $\mathbb P(C')$ and so the image of $s$ is disjoint from $\mathbb P(C')$ over a neighborhood of $X$. We define $t$ as $\rho\circ u \circ x$ to make the diagram commute. 

Because $e\circ t$ is proper and $e$ is separated, $t$ is proper. 
 
By definition, $\tilde{C}$ is the pushforward from $T^* X \times_X \tilde{X}$ to $T^* \tilde{X}$ of the restriction of $C$ from $T^* X$ to $T^* X \times_X \tilde{X}$. Because this pushforward and pullback are compatible with taking quotients by $\mathbb G_m$, we have \[  \mathbb P( \inhyper{C})  = (\rho \circ u) _*  ( \mathbb P(C') \times_X \inhyper{X}) \] and \[(\rho\circ u)^* \mathbb P(\inhyper{V}) =(b \circ u)^*  \mathbb P(V)\]
so \cite[Proposition 8.1.1(c)]{Fulton} \[ ( \mathbb P( \inhyper{C}), \mathbb P(\inhyper{V}))_{\mathbb P(T^* \inhyper{X}), x}   
=  (  \mathbb P(C') \times_X \inhyper{X}  , \mathbb P(V))_{\mathbb P(T^* X)\times_{X} \inhyper{X},x }  = (\mathbb P(C'), \mathbb P(V))_{\mathbb P(T^* X)} \] where the pullback along the open immersion $(b\circ u)$ does not affect the intersection number since the intersection locus is a closed subset of $  \mathbb P(T^*X \times_X \inhyper{X}  ) - s(\inhyper{X})$.

%
%
%
%
%

Finally we have \[ (\mathbb P(C'), \mathbb P(V))_{\mathbb P(T^* X) } =(\mathbb P(C), \mathbb P(V))_{\mathbb P(T^* X) }\] because the difference between $\mathbb P(C)$ and $\mathbb P(C')$ is the fiber over $x$, and because $Y$ has codimension at least one we can perturb $\mathbb P(V)$ to not intersect this fiber.

Next we check the strict transform condition. The exceptional fiber of the blowup of $ T^* X$ at the cotangent space at $x$ is $\mathbb P ((TX)_x ) \times \mathbb P ( (T^* X)_x)$.  Let $\omega \in (T^*X)_x$ be a conomal vector to $X^*$. Note that $\omega$ is a general vector in $(T^*X)_X$, $(TY)_x$ is general among all $\dim X-i$-dimensional subspaces of $(TX)_x$ perpendicular to $\omega$, and $V_x$ is general among $i+1$-dimensional vector subspaces of $(T^*X)_x$ containing $\omega$. Let $Z$ be the intersection of the strict transform of $\mathbb P(C)$ with the exceptional divisor. We must check that the intersection of $ \mathbb P (  (TY)_x) \times P (V_x)$ with $Z$ vanishes.

The dimension of $Z$ is $\dim X-2$ because $Z$ has dimension one less than $\mathbb P(C)$, which itself has dimension one less than $C$, which has dimension $\dim X$.
So it suffices to prove that for $(v_1,v_2) \in Z \subset \mathbb P ( (T^*X)_x) \times \mathbb P( (TX)_x)$, the codimension of the space of triples $( \omega, (TY)_x, V_x)$ such that $v_1\in (TY)_x , ,v_2 \in V_x $ inside the space of all triples $(\omega, (TY)_x, V_x)$ is at least $\dim X-1$.



The pairs  $(\omega,  V_x)$ such that $\omega, v_2 \in V_x$ have codimension $\dim X - (i+1)$ in the space of pairs $(\omega, V_x)$ with $ \omega \in V_x$. (We can ignore $\omega$ for this calculation).

To have $v_1 \in (TY)_x$ and $(TY)_x$ perpendicular to $\omega$, we must have $v_1 \cdot \omega =0$. This is a codimension $1$ condition on $(\omega, V_x)$. To check this, note that the space of pairs $\omega, V_x$ with  $\omega, v_2 \in V_x$ is irreducible, so it suffices to check the function $v_1 \cdot \omega$ is not identically zero. We can do this by choosing $\omega$ with $v_1 \cdot \omega$ nonzero and then choosing $V_x$ to contain $\omega$ and $v_2$, using $\dim V_x= i+1 \geq 2$ since we have assumed $i>0$.

Over each pair $(\omega, V_x)$ with  $\omega, v_2 \in V_x$ and $v_1 \cdot \omega = 0$,  the $\dim X-i$-dimensional spaces $(TY_x)$ that contain $v_1$ and are perpendicular to $\omega$ have codimension $i-1$ among all $\dim X-i$-dimensional spaces perpendicular to $\omega$, for a total codimension of $\dim X - 1$.

This completes the case $0 < i < \dim X$.

 Finally, for $i = \dim X$, observe that any conical cycle whose projection to the cotangent space at $\inhyper{X}$ is the zero section was already the zero section, and any cycle whose restriction to a general hypersurface is the zero section was already the zero section.

\end{proof}

Recall the statement of Theorem \ref{first-Massey-intro}:

\begin{theorem}\label{first-Massey}Let $X$ be a smooth variety over a perfect field $k$ and let $\ell$ be a prime invertible in $k$. Let $K$ be a perverse sheaf of $\mathbb F_\ell$-modules on $X$. 

Then $\dim_{\mathbb F_\ell} \mathcal H^{-i}(K)_x$ is at most $i$th polar multiplicity of $CC(K)$ at $x$.  \end{theorem}

\begin{proof} This is an \'{e}tale-local question, so we may assume that $X$ is a smooth projective variety by passing to an affine open subset and embedding into projective space, then extending $K$ to keep it perverse. In fact, we fix an embedding into projective space. 

This follows by induction on $i$. Let $p$ and $q$ be the map defined by a general pencil of conics, as in Diagram \ref{pq-diagram}. We have a distinguished triangle \[ p^* K \to  R \Psi_q p^* K \to R \Phi_q p^*  K  .\] Taking stalk cohomology at $x$, we have an exact sequence

\[ \mathcal H^{-i-1}( R \Phi_q p^*  K)_x \to \mathcal H^{-i} ( K)_x \to \mathcal H^{-i} ( R \Psi_q p^* K)_x \to  \mathcal H^{-i}( R \Phi_q p^*  K)_x .\]

Because $K$ is perverse, $p^*K$ is perverse in a neighborhood of $x$, and thus $R \Psi_q p^* K[-1]$ is perverse.

Thus for $i=0$, $ \mathcal H^0 ( R \Psi_q p^* K)_x $ vanishes and we have
\[ \dim \mathcal H^0(K)_x \leq \dim  (R^{-1}\Phi_q p^*K)_x \leq \operatorname{dimtot} (R^{-1}  \Phi_q p^* K)_x = -  \operatorname{dimtot} (R \Phi_q p^* K)_x \] which is at most the multiplicity of the cotangent space at $x$ in $CC(K)$ which by Lemma \ref{multiplicity-hypersurface} is the $0$th polar multiplicity of $CC(K)$ in $x$.

By Lemma \ref{final-nearby-formula}(4), $\mathcal H^{i-1}( R \Phi_q p^*  K)_x$ vanishes unless $i=0$, so the map \[ \mathcal H^i ( K)_x \to \mathcal H^i ( R \Psi_q p^* K)_x \] is injective unless $i=0$. Thus for $i>0 $, we have

\[ \dim \mathcal H^{-i} ( (p^*K)_x)  \leq \mathcal H^{-i} ( R \Psi_q p^* K)_x. \]

Because $R \Psi_q p^* K [-1]$ is perverse, we can apply the induction hypothesis, to see that $\dim  \mathcal H^{-i} ( R \Psi_q p^* K)_x$  is at most the $i-1$st polar multiplicity of $CC( R \Psi_q p^* K [-1])$ at $x$. By Lemma \ref{final-nearby-formula}(1), $CC( R \Psi_q p^* K [-1])$ is the projection to the cotangent space of the conic of the restriction to the conic of $CC'(K)$. By Lemma \ref{multiplicity-hypersurface}, the $i-1$st polar multiplicity of this is the $i$th polar multiplicity of $CC(K)$, verifying the induction step.\end{proof}

In characteristic zero, the inequality $(R\Phi_f K)_x \leq \operatorname{dimtot} (R \Phi_f K)_x$ would be an identity, and we could use the Morse inequalities to derive additional information about the Betti numbers of $K$, as Massey does in \cite[Corollary 5.5]{Massey}, but in our case the analogue of the Morse inequalities are unhelpful.

\begin{proof}[Proof of Corollary \ref{vanishing-intro}] In view of Theorem \ref{first-Massey} it suffices to check that for $i< \dim X - \dim (SS(K)_x)$, the $i$th polar multiplicity of $CC(K)$ at $x$ vanishes.  For $V$ a vector bundle of rank $i+1$, $\mathbb P(V)$ has dimension $i$ in the fiber of $0$, and $\mathbb P( CC(K))$, which is contained in $\mathbb P( SS(K))$, has codimension $\dim X - \dim (SS(K)_x)$ in the fiber at zero, so for a generic $V$ these do not intersect and their intersection number, which is the polar multiplicity, vanishes. \end{proof}

\section{Application to a conjecture of Shende and Tsimerman}
This section is devoted to proving \ref{theta-betti-bound}, following the strategy used by \citet{TS} to prove the characteristic zero analogue. To do this, we must first redo their calculation of the characteristic cycle in characteristic $p$, using Saito's definition of the characteristic cycle, and then explain why their estimate for the polar multiplicities of this cycle remains valid in characteristic $p$. 

While the argument is from a different perspective, and uses different notation in some parts, the ideas are essentially all due to Shende and Tsimerman. Because we are redoing the argument anyways, we take the opportunity to tighten up some of the inequalities.

Because the statement to prove is purely cohomological in nature, we work for simplicity over an algebraically closed field $k$. 

 Let $C$ be a smooth projective hyperelliptic curve over an algebraically closed field $k$, $\tau$ its hyperelliptic involution, $J$ its Jacobian, and $C^{(n)}$ the $n$th symmetric power, which we view as a moduli space of degree $n$ divisors.
 
 We can fix some degree $1$ divisor on $C$ which is equal to half the hyperelliptic class, and therefore identity the group of degree $n$ divisor classes on $C$ with $J$ for all $n$. In particular, once we have done this, for $P$ a point of $C$, the divisor class $[P+ \tau(P)]$ will equal the hyperelliptic class and thus vanish.

Let $\underline{A}^{a,b}: C^{(g-a)} \times C^{(g-b)} \to J$ be the map sending a pair $(D_1,D_2)$ of divisors to the divisor class $[D_1+D_2]$.

Note that the cotangent bundle of $J$ is a trivial bundle, and we can identify its fiber at any point as the vector space $H^0(C, K_C)$.  
  
 For natural natural numbers $w_1,w_2$ with $w_1+ w_2 \leq g$, consider the closed subset $Z_{w_1,w_2}$ of  $C^{ (w_1) } \times C^{ (w_2)} \times H^0(C, K_C) $ consisting of pairs $(D_1,D_2, \omega )$ with $D_1$ a divisor of degree $w_1$, $w_2$ a divisor of degree $w_2$, and $\omega$ a differential form on $C$ whose divisor of zeroes is at least $D_1+D_2+ \tau(D_1) + \tau(D_2)$. 
 
 Note that, because $K_C (- D_1- D_2-  \tau(D_1) - \tau(D_2) )$ is the pullback from $\mathbb P^1$ of a divisor of degree $g-1 -w_1-w_2$, \[\dim H^0 (C, K_C (- D_1- D_2-  \tau(D_1) - \tau(D_2) ) ) = g-w_1-w_2,\] and so $Z_{w_1,w_2}$ is smooth of dimension $g$.
  
 We can define a map $pr_{W_1,W_2}: Z_{w_1,w_2} \to T^* J$ by sending $(D_1,D_2, \omega)$ to $( [D_1+ 2D_2], \omega)$. Because $pr_{W_1, W_2}$ is proper, $pr_{w_1,w_2 *} [ Z_{w_1,w_2} ] $ is an algebraic cycle of codimension $g$ on $T^* J$. 

\begin{lemma}\label{ts-singular} The pushforward $ \underline{A}^{A,b}_{\circ} (C^{(g-a)} \times C^{(g-b)} ) $ of the zero-section of $T^* ( C^{(g-a)} \times C^{(g-b)} )$ is contained in the union of the zero section of $T^* J$ with the union over $w_1,w_2$ such that $w_1+w_2<g$ of the support of $pr_{W_1,W_2 *} [ Z_{w_1,w_2} ]$. \end{lemma}

\begin{proof} The inverse image $\left( d\underline{A}^{a,b}\right)^{-1}  (C^{(g-a)} \times C^{(g-b)} )$ is the set of pairs $(D_a, D_b, \omega)$ with $D_a$ a divisor of degree $g-a$, $D_b$ a divisor of degree $g-b$, and $\omega\in H^0(C, T^*C)$ such that $d\underline{A}^{a,b} (D_a,D_b) (\omega)$ vanishes, The pushforward of $\left( d\underline{A}^{a,b}\right)^{-1}   (C^{(g-a)} \times C^{(g-b)} )$ to $T^* J$ is the set of all pairs $( [D_a + D_b], \omega)$ where $[D_a + D_b]$ is the divisor class of $D_a+D_b$, such that $d\underline{A}^{a,b} (D_a,D_b) (\omega)$ vanishes. By Definition \ref{circ-forward}, this pushforward is $ \underline{A}^{A,b}_{\circ} (C^{(g-a)} \times C^{(g-b)} ) $.

Let us fix $(D_a,D_b,\omega)$ such that $d\underline{A}^{a,b} (D_a,D_b) (\omega)$ vanishes. We will show that $( [D_a + D_b], \omega)$ is contained in either $pr_{W_1,W_2 *} [ Z_{w_1,w_2} ]$ for some $w_1,w_2$ or the zero section.

We can represent the tangent space of  $C^{(a)}$ at $D_a$ as $H^0(C, \mathcal O(D_a) /\mathcal O)$ so that by Serre duality the cotangent space is $H^0(C, K_C / K_C(-D_a) )$. Then the derivative map \[ H^0(C, K_C) \to H^0(C, K_C / K_C (-D_a) ) \oplus H^0(C, K_C / K_C (-D_b) )\] is given by reducing a section modulo $D_a$ and $D_b$. Hence $\omega$ is in the kernel of $d\underline{A}^{a,b} (D_a,D_b) $ if and only if its divisor is greater than or equal to $D_a$ and also greater than or equal to $D_b$. Thus the divisor of $\omega$ is greater than or equal to $\max(D_a,D_b) $

Let $D' $ be obtained from $D_a+D_b$ by iteratively subtracting divisors of the form $[P + \tau (P)]$ until it is no longer possible to subtract divisors of the form $[P + \tau(P)]$ from $D'$ while keeping it effective. This means that there is no point $P$ for which $P$ and $\tau(P)$ are both in the support of $D'$, except possibly for points fixed by $\tau$, which must have multiplicity at most $1$. Let $D_1$ be the sum of all the points with odd multiplicity in $D'$ and let $D_2 = (D' -D_1)/2$.  Then by construction \[  [D_a + D_b] = [D' ] = [D_1 + 2D_2].\] 

Next, let us check that the divisor of $\omega$ is at least $D_1+ D_2 + \tau(D_1) + \tau(D_2) $

To do this, consider a point $P$ in the support of $D_1+ D_2 + \tau(D_1) + \tau(D_2) $ and let $m$ be the multiplicity of $D_1+ D_2 + \tau(D_1) + \tau(D_2) $ at $P$.

If $P$ is fixed by $\tau$, we can have $m$ at most $2$, and $m=0$ unless $D'$ vanishes at $P$. If $D'$ vanishes at $p$, then $D_a$ or $D_b$ vanishes at $P$, which means $\omega$ vanishes at $P$. Then $\omega$ must vanish to order $2$ at $P$ because global $1$-forms on a hyperelliptic curve are negated by the hyperelliptic involution and so vanish to even order at hyperelliptic points. So in either case, $\omega$ vanishes to order at least $m$ at $p$.

Otherwise, we cannot have both $P$ and $\tau(P)$ in the support of $D_1+D_2$, so either $P$ or $\tau(P)$ has multiplicity $m$ in $D_1+D_2$. Because the divisor of $\omega$ is symmetric, without loss of generality we can assume $P$ has multiplicity $m$ in $D_1+D_2$. Because the multiplicity of $D_1$ at $P$ is at most $1$, the multiplicity of $D_2$ is at least $m-1$, so the multiplicity of $D' = D_1+2 D_2$ is at least $2(m-1)+1=2m-1$. Thus the multiplicity of $D_a+D_b$ at $P$ is at least $2m-1$. Then the multiplicity of either $D_a$ or $D_b$ must be at least $\left\lceil \frac{2m-1}{2} \right\rceil =m$. Thus $\omega$ vanishes to order $m$ at $P$.

So in either case the order of vanishing of $\omega$ at $P$ is at least $m$, as desired.

So we have shown that the divisor of $\omega$ is at least $D_1+ D_2 + \tau(D_1) + \tau(D_2) $ and thus $(D_1,D_2, \omega)$ is a point of $Z_{w_1, w_2}$ where $w_1 = \deg D_1$ and $w_2 = \deg w_2$. Because $[D_1 +2D_2 ]= [D_a+D_b]$, it follows that  $( [D_a+ D_b] ,\omega) $ is the image of $(D_1,D_2, \omega)$ under $ pr_{w_1,w_2 }$.  Finally, note that if $\deg D_1 + D_2 \geq g$, then the divisor of $\omega$ is at least a divisor of degree at least $2g$. Thus $\omega$ vanishes and so $([D_a+D_b], \omega)$ is contained in the zero section. So $([D_a+D_b],\omega)$ is contained in either $pr_{w_1,w_2} (Z_{w_1,w_2})$ for $w_1+w_2<g$ or the zero section.

\end{proof}

\begin{lemma}\label{ts-smooth-cc} The characteristic cycle $CC ( \underline{A}_{a,b *} \mathbb Q_\ell [ 2g-a-b] )$ is equal to \[ \sum_{0 \leq w_1+ w_2 < g} m_{w_1,w_2,a,b}  pr_{w_1,w_2 *} [ Z_{w_1,w_2} ]  + m_{a,b} [ A]  \]

where $m_{w_1,w_2,a,b}$ is the coefficient of $v_1^{g-a} v_2^{g-b}$ in  \[ (v_1+ v_2+ v_1^2 v_2 + v_1 v_2^2 )^{w_1} (v_1v_2)^{w_2} ( 1  + v_1^2 + 2v_1 v_2 + v_2^2 + v_1^2 v_2^2 )^{g-1-w_1 -w_2} \]
and \[ |m_{a,b} | \leq 8^g .\]

 \end{lemma}

\begin{proof} By definition, $CC ( \mathbb Q_\ell [ 2g-a-b] )$ is equal to the zero-section. Hence by Lemma \ref{ts-singular},  $\underline{A}^{A,b}_{\circ} ( SS (\mathbb Q_\ell) ) $ is contained in the union of the zero section with $pr_{W_1,W_2 *} [ Z_{w_1,w_2} ]$ for $w_1+w_2<g$, and thus has dimension $\leq g$.  This verifies condition (2.20) of \cite[Theorem 2.2.5]{saito-direct}. The other conditions (that $J$ is projective, that $\underline{A}^{A,b}$ is quasi-projective and proper on the support of $\mathbb Q_\ell$, and that $\mathbb Q_\ell$ is constructible) are clear. Hence from \cite[Theorem 2.2.5]{saito-direct} we deduce \[CC  ( \underline{A}_{a,b *} \mathbb Q_\ell[2g-a-b] )= \underline{A}^{a,b}_! CC( \mathbb Q_\ell[2g-a-b]) = \underline{A}^{a,b}_! [ C^{(g-a)} \times C^{(g-b) }]. \]

To prove \[  \underline{A}^{a,b}_! [ C^{(g-a)} \times C^{(g-b) }] = \sum_{0 \leq w_1+ w_2 < g} m_{w_1,w_2,a,b}  pr_{w_1,w_2 *} [ Z_{w_1,w_2} ]  + m_{a,b} [ A]  ,\] let us first prove that the two sides become equal after we pull back by a general section $A \to T^* A$ coming from a general element $\omega  \in H^0(C, K_C)$. 

By a push-pull formula, the pullback of $\underline{A}^{a,b}_! [ C^{(g-a)} \times C^{(g-b) }]$ along $\omega$ is simply the pushforward along $\underline{A}^{a,b}$ of the pullback of $[ C^{(g-a)} \times C^{(g-b) }]$ along $d \underline{A}^{a,b} (\omega)$, which is the pushforward along $\underline{A}^{a,b}$ of the zero locus of $d \underline{A}^{a,b} (\omega)$. Because $\omega$ is general, it has $2g-2$ distinct zeroes forming $g-1$ orbits of size $2$ under $\tau$. Let $x_1,\dots, x_{g-1}, x_g,\dots, x_{2g-2}$ be these zeroes with $x_{g-1+i} = \tau(x_i)$. 

It follows that $(D_a, D_b)$ lies in the zero locus of $d \underline{A}^{a,b} (\omega)$ if and only if $D_a$ is the sum of a subset of size $a$ of these zeroes and $b$ is the sum of a subset of size $b$ of these zeroes. Furthermore the multiplicities of each of these pairs in the zero locus of  $d \underline{A}^{a,b} (\omega)$ must be one, as the sum of all the multiplicities must equal the topological Euler characteristic $\binom{2g-2 }{ a} \binom{2g-2 }{ b}$ of $C^{(a)} \times C^{(b)}$. Thus we can write \[ \omega^* ( \underline{A}^{a,b}_! [ C^{(g-a)} \times C^{(g-b) } ] ) = \sum_{ \substack{ S ,T \subseteq \{x_1,\dots, x_{2g-2}\} \\ |S| = a, |T|= b}} \left[ \sum_{x_i \in S} x_i + \sum_{x_i \in T} x_i \right] .\]

On the other hand, $\omega^*  pr_{w_1,w_2 *} [ Z_{w_1,w_2} ]$ is simply the pushforward from $C^{(w_1) } \times C^{(w_2)}$ to $J$ along the map $(D_1,D_2) \to [D_1+2D_2]$ of the set of $(D_1,D_2)$ with $|D_1|=w_1$, $|D_2|=w_2$, and such that $ D_1 + D_2 + \tau(D_1) + \tau(D_2)$ is at most the divisor of $\omega$.  This occurs when $D_1$ is a subset of $\{x_1,\dots, x_{2g-2}\}$ of size $w_1$, $D_2$ is a subset of  $\{x_1,\dots, x_{2g-2}\}$ of size $w_2$, and $D_1, D_2, \tau(D_1), \tau(D_2)$ are all disjoint. 

To match the two sides, we choose for each $S, T$ a pair $D_1, D_2$ such that $ \sum_{i \in S} x_i + \sum_{ i \in T} x_i  = D_1 + 2D_2$ and $D_1, D_2$ satisfy the stated conditions. To do this, observe that for each $i$ from $1$ to $g-1$, the linear combination of indicator functions
\[ 1_{ x_i \in S}+ 1_{x_i \in T} - 1_{x_{i+g-1} \in S} -1_{x_{i+g-1} \in T}\] takes the value $2, 1,0,-2,$ or $2$. If it is $2$, put $x_i$ in $D_2$. If it is $1$, put $x_i$ in $D_1$. If it is $-1$, put $x_{i+g-1}$ in $D_1$. If it is $-2$, put $x_{i+g-1}$ in $D_2$. If it is $0$, put neither $x_{i}$ nor $x_{i+g-1}$ in $D_1$ or $D_2$.

To prove the two pullbacks are equal, it suffices to prove that for any $(D_1,D_2)\subseteq \{x_1,\dots,x_{2g-2} $ with  $|D_1|=w_1$, $|D_2|=w_2$, and $D_1,D_2,\tau(D_1), \tau(D_2)$ all disjoint, the number of $S,T$ with $|S|=a, |T|=b$, where this process produces $D_1,D_2$, is $m_{w_1,w_2,a,b}$.  We use the standard generating functions approach to counting the number of ways to make a series of independent choices subject to linear constraints:

For any pair $(x_i,\tau(x_i))$, if $x_i \in D_2$, the tuple $(1_{x_i\in S}, 1_{x_i \in T}, 1_{\tau(x_i) \in S}, 1_{\tau(x_i)\in T})$ can only take the value $(1,1,0,0)$. We assign this value the term $v_1v_2$.

If $x_i \in D_1$, the tuple must take one of the four values $(1,0,0,0), (0,1,0,0), (1,1,1,0),(1,1,0,1)$. We assign these values the terms $v_1,v_2,v_1^2 v_2, $ and $v_1v_2^2$ respectively.

For $x_i \not \in D_1,x_i \not\in D_2, \tau(x_i)\not\in D_1, \tau(x_i)\not\in D_2$, the tuple must take one of the six values $(0,0,0,0),(1,0,1,0), (1,0,0,1),(0,1,1,0),(0,1,0,1),(1,1,1,1)$. We assign these values the terms $1,v_1^2, v_1v_2,v_1v_2, v_2^2,$ and $v_1^2v_2^2$ respectively.

Then by our system of assignments, each choice of $S,T$ where this process produces $D_1,D_2$ corresponds to a monomial in \[ (v_1+ v_2+ v_1^2 v_2 + v_1 v_2^2 )^{w_1} (v_1v_2)^{w_2} ( 1  + v_1^2 + 2v_1 v_2 + v_2^2 + v_1^2 v_2^2 )^{g-1-w_1 -w_2} \] and the ones with $|S| = a, |T|=b$ are exactly the monomials $v_1^a v_2^b$, so the coefficient of $v_1^a v_2^b$ is the number of $S,T$, as desired.

Now because the two cycles agree when pulled back to a general fiber of the projection to $H^0(C, K_C)$, they are equal modulo a sum of irreducible components whose projection to $H^0(C, K_C)$ is not dense. Because $ \underline{A}^{a,b}_! [ C^{(g-a)} \times C^{(g-b) }]$ is contained in $ \underline{A}^{a,b}_{\circ} ( C^{(g-a)} \times C^{(g-b) })$, by Lemma \ref{ts-singular}, these irreducible components must be contained in either $pr_{W_1,W_2} (Z_{w_1,w_2})$ or the zero section. Because $Z_{w_1,w_2}$ is irreducible of dimension at most $g$, the same properties hold for $pr_{W_1,W_2} (Z_{w_1,w_2})$, and so these irreducible components must equal either $pr_{W_1,W_2} (Z_{w_1,w_2})$ or the zero section. Because the projection of $Z_{w_1,w_2}$ to $H^0(C,K_C)$ is dense, the problematic  components cannot be $pr_{W_1,W_2} (Z_{w_1,w_2})$, so they must be the zero section.

To calculate the multiplicity of the zero-section in $CC ( \underline{A}_{a,b *} \mathbb Q_\ell [ 2g-a-b] )$, we notice that it is equal by definition to $(-1)^g$ times the Euler characteristic of the stalk of $\underline{A}_{a,b *} \mathbb Q_\ell [ 2g-a-b]$ at the generic point, which is $(-1)^{g+ a+b}$ times the topological Euler characteristic of the generic fiber of $\underline{A}_{a,b}$. This Euler characteristic is bounded in \citep[ Proposition 3.16]{TS} as at most $8^g$, giving our stated formula.

Note that when calculating this Euler characteristic, it does not matter if we work in characteristic zero or characteristic $p$, as we can lift everything in sight to characteristic zero, and the Euler characteristic is preserved by this lifting. \end{proof} 

We can factor $\underline{A}^{a,b}$ as the composition $mult \circ (\pi^a \times \pi^b)$ where $mult: J \times J \to J$ is the multiplication and $\pi_n : C^{(n)}\to J$ sends a divisor to its class. Let $\Theta_n$ be the image of $C^{(n)}$ under $\pi_n$, i.e. the set of degree $n$ divisor classes which are effective, and let $i^n$ be the inclusion of $\Theta_n$ into $J$.

Let $\Sigma^{a,b} = mult \circ (i^a \times i^b) $.

\begin{lemma}\label{ts-decomposition} For $0 \leq n \leq g$, \[ \pi^n_* \mathbb Q_\ell= \bigoplus_{r=0}^{n/2} i^{n+2r}_* \mathbb Q_\ell [-2r] (-r) .\] \end{lemma}

\begin{proof} This is obtained as part of the proof of \cite[Lemma 2.9]{IY} or \cite[Lemma 3.1]{TS}. \end{proof}

\begin{lemma}\label{ts-singular-cc}  The characteristic cycle $CC ( \Sigma^{a,b}_* \mathbb Q_\ell[2g-a-b] )$ is equal to \[ \sum_{0 \leq w_1+ w_2 < g} m'_{w_1,w_2,a,b}  pr_{w_1,w_2 *} [ Z_{w_1,w_2} ]  + m'_{a,b} [ A]  \]

where $m_{w_1,w_2,a,b}$ is the coefficient of $v_1^{g-a} v_2^{g-b}$ in \[ (1 -v_1^{-2} ) (1-v_2^{-2} ) (v_1+ v_2+ v_1^2 v_2 + v_1 v_2^2 )^{w_1} (v_1v_2)^{w_2} ( 1  + v_1^2 + 2v_1 v_2 + v_2^2 + v_1^2 v_2^2 )^{g-1-w_1 -w_2} \]

and \[|m'_{a,b}| \leq 4 \cdot 8^g.\] \end{lemma}

\begin{proof} We have \[ \underline{A}^{a,b}_* \mathbb Q_\ell = mult_*  (\pi^a \times \pi^b)_* \mathbb Q_\ell = mult_*\left(  \pi^a_* \mathbb Q_\ell\boxtimes \pi^b_* \mathbb Q_\ell \right) \] \[= mult_* \left( \left( \bigoplus_{r=0}^{n/2-a} i^{a+2r}_* \mathbb Q_\ell [-2r] (-r)  \right) \boxtimes \left( \bigoplus_{s=0}^{n/2-b} i^{b+2s}_* \mathbb Q_\ell [-2s] (-s)  \right)\right) \] \[= \bigoplus_{r=0}^{n/2-a} \bigoplus_{s=0}^{n/2-b}  mult_* (i^{a+2r} \times i^{b+2s})_* \mathbb Q_\ell [-2r-2s](-r-s) =\bigoplus_{r=0}^{n/2-a} \bigoplus_{s=0}^{n/2-b}  \Sigma^{a+2r, b+2s}_* \mathbb Q_\ell  [-2r-2s](-r-s)\] and thus
\[ CC( \underline{A}^{a,b}_* \mathbb Q_\ell ) = \sum_{r=0}^{n/2-a} \sum_{s=0}^{n/2-b} CC(  \Sigma^{a+2r, b+2s}_* \mathbb Q_\ell) \]
and so solving for $CC ( \Sigma^{a,b}_* \mathbb Q_\ell)$ we get \[ CC ( \Sigma^{a,b}_* \mathbb Q_\ell) = CC( \underline{A}^{a,b}_* \mathbb Q_\ell ) - CC( \underline{A}^{a+2,b}_* \mathbb Q_\ell )-CC( \underline{A}^{a,b+2}_* \mathbb Q_\ell )-CC( \underline{A}^{a+2,b+2}_* \mathbb Q_\ell )\] and then the claim follows from Lemma \ref{ts-smooth-cc}. \end{proof}

\begin{lemma}\label{individual-polar-bound} For a line bundle $E$ in $J$ and $0\leq i < g-1$, the $i$th polar multiplicity of $pr_{w_1,w_2 *} [ Z_{w_1,w_2} ]$ at $E$ is at most

 \[ 2^{w_1 +w_2} { g \choose i} \sum_{\substack{c+d = w_1+w_2 -i \\ c \leq w_1, d\leq w_2 }} {g-i-1 \choose c,d,g-1-w_1-w_2} 2^{w_2-d} {w_1 + w_2 -c -d \choose w_1 -c }.\]    \end{lemma}

\begin{proof}  We apply Definition \ref{polar-multiplicity}. 

Let us take $L \subseteq  H^1(C, \mathcal O_C)$ a generic subspace of rank $g-i-1$, and $L^\vee \subseteq H^0(C, K_C)$ its perpendicular space of dimension $i+1$. Let $V$ on $J$ be the constant vector bundle $L^\vee$. Then $V_x = L^\vee$ is a generic subspace. By Definition \ref{polar-multiplicity}, the $i$th polar multiplicity of $pr_{w_1,w_2 *} [ Z_{w_1,w_2} ]$ at $x$ equals the multiplicity of \[ \pi_* ( \mathbb P ( pr_{w_1,w_2 *} [ Z_{w_1,w_2} ] ) \cap \mathbb P(V)) \] at $x$.

Let $ \mathbb P ( Z_{w_1,w_2} )$ be the moduli space of triples $(D_1,D_2,\omega)$ with $D_1 \in C^{(w_1)}, D_2\in C^{(w_2)}, \omega \in \mathbb P ( H^0(C, K_C))$ and let $pr_{w_1,w_2}' $ be the projection to $\mathbb P (T^* J)$ sending $(D_1,D_2,\omega)$ to $([D_1+2D_2],\omega)$. Then we have \[ \mathbb P ( pr_{w_1,w_2 *} [ Z_{w_1,w_2} ] ) = pr'_{w_1,w_2 *} [\mathbb P(Z_{w_1,w_2} )] \] so 
\[ \pi_* ( \mathbb P ( pr_{w_1,w_2 *} [ Z_{w_1,w_2} ] ) \cap \mathbb P(V)) = \pi_* (  pr'_{w_1,w_2 *} [\mathbb P(Z_{w_1,w_2} )] ) \cap \mathbb P(V) =( \pi \circ pr'_{w_1,w_2})_* (  pr_{w_1,w_2}^{'*} \mathbb P(V) ).\]

We can view the cycle $\mathbb P(V)$ as the pullback of $\mathbb P(L^\perp)$ from $\mathbb P(H^0(C,K_C))$, so we can view $pr_{w_1,w_2}^{'*} \mathbb P(V)$ as the pullback of $\mathbb P(L^\perp)$ under the projection $\sigma: \mathbb P ( Z_{w_1,w_2} )\to \mathbb P(H^0(C,K_C))$.

For $\omega \in H^0(C, K_C)$, $\div(\omega)$ is the pullback from $(C/\tau)= \mathbb P^1$ of a divisor on $\mathbb P^1$, so $\div(\omega)- D_1 -\tau(D_1) - D_2 -\tau(D_2)$ is the pullback from $(C/\tau)$ of a divisor of degree $g-1-w_1 -w_2$. This divisor uniquely determines $\omega$. This gives an isomorphism between $\mathbb P ( Z_{w_1,w_2}) $ and $C^{(w_1)} \times C^{(w_2)} \times (C/\tau)^{g-1-w_1-w_2}$. Under this interpretation, the map $\pi' \circ  pr'_{w_1,w_2}$ is equal to the map $\pi_{w_1,w_2}:  C^{(w_1)} \times C^{(w_2)} \times (C/\tau)^{g-1-w_1-w_2} \to J$ that sends $(D_1,D_2, D_3) $ to $ [D_1+2D_2] $. Furthermore, under this interpretation, the map to $\mathbb P ( H^0(C, K_c)) = \mathbb P ( C/\tau, \mathcal O(g-1)) = (C/\tau)^{g-1}$ may be obtained by projecting $D_1$ and $D_2$ to $ C/\tau$ and then adding all three divisors, as the pullback of this sum to $C$ is necessarily $\div(\omega)$.

Shende and Tsimerman define a polar variety $P_L'V_{w_1,w_2} = \pi_{w_1,w_2} (\sigma^{-1} (\mathbb P(L^\vee)))$ using exactly this definition of $\pi_{w_1,w_2}$ and $\sigma$ (except that they use the letters $r$ and $s$ instead of $w_1$ and $w_2$. )

They calculated \cite[Lemma 3.22]{TS} that the cycle class of  $P_L'V_{w_1,w_2}$ equals
\[    \sum_{\substack{c+d = w_1+w_2 -i \\ c \leq w_1, d\leq w_2 }} 2^{ w_1+w_2-i}  {g-i-1 \choose c,d,g-1-w_1-w_2} 2^{w_2-d} {w_1 + w_2 -c -d \choose w_1 -c }  [\Theta_i].\]   
(In fact they use slightly different notation - to obtain their formula, substitute $r$ for $w_1$, $s$ for $w_2$, $a$ for $c$, $b$ for $d$, and $g-1-k$ for $i$.)

Because we can lift everything smoothly to characteristic zero, it does not matter here whether we do intersection theory in characteristic zero or characteristic $p$.

Now applying Theorem \ref{appendix}, we see that the multiplicity is at most \[\sum_{\substack{c+d = w_1+w_2 -i \\ c \leq w_1, d\leq w_2 }} 2^{ w_1+w_2-i}  {g-i-1 \choose c,d,g-1-w_1-w_2} 2^{w_2-d} {w_1 + w_2 -c -d \choose w_1 -c }   \] times \[  2^{i-1} ([\Theta_i], [\Theta_{g-i}] )= 2^{i-1} {g \choose i} \]
(or times $1$ if $i=0$) which is at most  \[ 2^{w_1 +w_2} { g \choose i} \sum_{\substack{c+d = w_1+w_2 -i \\ c \leq w_1, d\leq w_2 }} {g-i-1 \choose c,d,g-1-w_1-w_2} 2^{w_2-d} {w_1 + w_2 -c -d \choose w_1 -c },\] where we ignore the factor of $2^{-1}$ in the case $i>0$ for simplicity.

\end{proof}

\begin{lemma}\label{global-polar-bound} For $x \in J$, the sum for $i$ from $0$ to $g-1$ of the $i$th polar multiplicity of $CC ( \Sigma^{a,b}_* \mathbb Q_\ell[2g-a-b] )$ at $x$ is at most $28^g/16$. \end{lemma}

\begin{proof} 
In the formula of Lemma \ref{individual-polar-bound} note that, because  $c+d= w_1+w_2-i$, we have \[  {g-i-1 \choose c,d,g-1-w_1-w_2} = {w_1+w_2 -i  \choose d} {g-i-1 \choose w_1 + w_2 -i} \] and \[ \sum_{\substack{c+d = w_1+w_2 -i \\ c \leq w_1, d\leq w_2 }}{w_1+w_2 -i  \choose d}2^{w_2-d} {w_1 + w_2 -c -d \choose w_1 -c } \] is the coefficient of $u^{w_2}$ in  $(1+2u)^i  (1+u)^{w_1+w_2-i}$. 

It follows that the $i$th polar multiplicity of $pr_{w_1,w_2 *} [ Z_{w_1,w_2} ] $ at $x$  is at most the coefficient of  $u^{w_2} $ in \[(1+2u)^i  (1+u)^{w_1+w_2-i}  2^{w_1+ w_2}  {g \choose i}    {g-i-1 \choose w_1 + w_2 -i}. \]

Let us bound $m'_{w_1,w_2,a,b}$ more crudely. It is at most $m_{w_1,w_2,a,b}$, which being the coefficient of $v_1^{g-a} v_2^{g-b}$  in\[  (v_1+ v_2+ v_1^2 v_2 + v_1 v_2^2 )^{w_1} (v_1v_2)^{w_2} ( 1  + v_1^2 + 2v_1 v_2 + v_2^2 + v_1^2 v_2^2 )^{g-1-w_1 -w_2} ,\] is at most the value of that polynomial at $1$, or  $4^{w_1} 6^{g-1-w_1-w_2}$. 

So the $i$'th polar multiplicity of \[ m'_{w_1,w_2,a,b} pr_{w_1,w_2 *} [ Z_{w_1,w_2} ] \] at $x$ is at most the coefficient of $u^{w_2}$ in  \[ 6^{g-1-w_1 -w_2} (4+2u)^i  (4+u )^{w_1+w_2-i} 2^{w_1+w_2}  {g \choose i} {g-i-1 \choose w_1 + w_2 -i}. \]

Now letting $w=w_1+w_2$, we can sum over all $w_2$ with  $0 \leq w_2 \leq w$, getting that the $i$th polar multiplicity of \[ \sum_{\substack{0 \leq w_1,w_2\\ w_1+w_2 =w } }  m'_{w_1,w_2,a,b} pr_{w_1,w_2 *} [ Z_{w_1,w_2} ]\] at $x$ is at most \[ 6^{g-1-w} 6 ^i  5^{w-i} 2^w   {g \choose i}   {g-i-1 \choose w-i} =  6^{g-1-w} 12 ^i  10^{w-i}    {g \choose i}   {g-i-1 \choose w-i} . \]

Now observe that \[\sum_{w=i}^{g-1} 6^{g-1-w} 6^i  10^{w-i}  {g \choose i}    {g-i-1 \choose w-i} =  12^i {g\choose i}  ( 6 + 10)^{ g-i-1} \] so we get the $i$th polar multiplicity of \[ \sum_{\substack{ 0\leq w_1,w_2 \\ w_1+w_2 \leq g} }m'_{w_1,w_2,a,b} pr_{w_1,w_2 *} [ Z_{w_1,w_2} ]\] is at most \[   {g \choose i} 12^i (16)^{g-i-1}.\] This is also a bound for the $i$th polar multiplicity of $CC  ( \Sigma^{a,b}_* \mathbb Q_\ell[2g-a-b] )$ by Lemma \ref{ts-singular-cc} and the fact that the zero-section does not contribute to the $i$th polar multiplicity for $i$ from $0$ to $g-1$. 

Now summing over $i$ from $0$ to $g-1$ and adding back in the $i=0$ term, we see that the sum from $i=0$ to $g-1$ of the polar multiplicities at $x$ is at most $ ( 12+ 16)^{g} / 16 = 28^g/ 16$.\end{proof} 

\begin{proof}[Proof of Theorem \ref{theta-betti-bound}]  By the proper base change theorem, \[ H^i (  (\Theta_{g-a} \cap L - \Theta_{g-b} )_{\overline{k}}, \mathbb Q_\ell)\] is the $i-2g+ab$th cohomology of the stalk at the point $L \in J$ of $ \Sigma^{a,b}_* \mathbb Q_\ell[2g-a-b]$. By \citep[Lemma 3.1 and Lemma 3.6]{TS}, $\Sigma^{a,b}_* \mathbb Q_\ell[2g-a-b] $ splits into a sum of its perverse homology sheaves, and ${}^p \mathcal H^i ( \Sigma^{a,b}_* \mathbb Q_\ell[2g-a-b] )$ is a constant sheaf of rank $\dim H^{g+i} (J,\mathbb Q_\ell$ for $i \neq 0$. Hence the sum of the Betti numbers of the stalk of  $ \Sigma^{a,b}_* \mathbb Q_\ell[2g-a-b] )$
at $L$ is at most the sum of the Betti numbers of the stalk of \[ {}^p \mathcal H^0 ( \Sigma^{a,b}_* \mathbb Q_\ell[2g-a-b] )\] plus the sum of the Betti numbers of $J$, and the second term is at most $4^g$.

Because the higher and lower perverse cohomology are constant, we have \[ CC \left( {}^p \mathcal H^0 ( \Sigma^{a,b}_* \mathbb Q_\ell[2g-a-b] )\right) = CC \left( \Sigma^{a,b}_* \mathbb Q_\ell[2g-a-b]  \right) - \sum_{0 \leq i \leq 2g, i\neq g} (-1)^{i+g } {2g \choose i} [ J] \] as the characteristic cycle of a constant sheaf is simply the zero section.

We apply Theorem \ref{first-Massey-intro} to $( {}^p \mathcal H^0 ( \Sigma^{a,b}_* \mathbb Q_\ell[2g-a-b] $. We get that the sum of its Betti numbers at $L$ is at most the sum of its polar multiplicities. All the polar multiplicities except the $g$th one match $CC \left( \Sigma^{a,b}_* \mathbb Q_\ell[2g-a-b]  \right) $, and thus their sum is bounded by $28^g/16$ by Lemma \ref{global-polar-bound}. The $g$th polar multiplicity is simply the multiplicity of the zero section, which is bounded by $4 \cdot 8^g + 4^g $ by Lemma \ref{ts-singular-cc} and the preceding formula.

So in total, the sum of the Betti numbers of $(\Theta_{g-a} \cap L - \Theta_{g-b} )_{\overline{k}}$ is bounded by $28^g/16 + 4 \cdot 8^g + 2\cdot 4^g$, as stated.

 \end{proof}

\appendix

\section[Bounding Multiplicity in Jacobians of Hyperelliptic curves]{ Bounding Multiplicity in Jacobians of Hyperelliptic Curves\\ \vspace{5pt} Jacob Tsimerman}

The purpose of this appendix is provide an upper bound for the multiplicity of a subvariety of the Jacobian of a curve in terms of intersection theory. Of particular importance to
us is that the method works in any characteristic. We therefore work over an arbitrary algebraically closed field $k$.

Let $C/k$ be a curve of genus $g$ and gonality $r$ so that there is a map $\pi:C\ra \PP^1$ of degree $r$. We call a point $P\in C(k)$ \emph{ordinary}
if $\pi$ is not ramified over $\pi(P)$. We similarly call an effective divisor $D=\sum_i P_i$ ordinary if all the $P_i$ are ordinary and the sets $\pi^{-1}(\pi(P_i))$ are all distinct. 
We set  $J^{(k)}$ to be the degree $k$ component of the Jacobian, and $\Theta_k\subset J^{(k)}$ the image of $\Sym^kC$. By translating we 
may non-canonically identify all the $J^{(k)}$, and thus canonically identify their Chow groups, which we do.

Now, consider the maps $\psi_k:\Sym^k C \ra \Sym^k\PP^1 \cong \PP^k, \phi_k:\Sym^k \ra J^{(k)}$, where the map $\psi_k$ is induced by $\pi$. Note that $\psi_k$ is finite, flat of degree 
$r^k$ and $\phi_k$ is proper. Thus we have corresponding maps on Chow groups \[CH^j(\PP^k)\xra{\psi_k^*} CH^j(\Sym^k C)\xra{\phi_{k*}} CH^j(J^{(k)}).\] 

Our goal is to prove the following

\begin{theorem}\label{appendix}

Let $V\subset J^{(g)}$ be a subvariety of codimension $j<g$. Then the multiplicity of $V$ at any point $v\in V(k)$ is bounded above by $r^{g-1-j} ([V],[\Theta_j])$. 

\end{theorem}  

\begin{proof}

We first prove the following

\begin{lemma}

Let $\ell_j$ be the class of a linear subspace of dimension $j$ in $\PP^k$. Then $\phi_{k*}\circ\psi_k^*L_j = r^{k-j} [\Theta_j]$. 

\end{lemma}

\begin{proof}

Let $D$ be an effective ordinary divisor of degree $k-j$, and let $L\subset\PP^k$ denote
$\psi_k(D+\Sym^j C)$. It is easy to see that $L$ is a linear subspace. Now $\psi_k^*L$ consists of the union of $D'+\Sym^jC$ where $D'$ varies over the $r^{k-j}$ divisors
consisting of  sums of $k-j$ points, including exactly one element form each set $\pi^{-1}(\pi(P_i))$. Moreover, since the map $\psi_k$ is etale over a generic point
of $L$, there is no generic multiplicity. Thus, $\psi_k^*\ell_j = r^{k-j} [D+\Sym^j C]$. 

Next, note that the scheme-theoretic image of $D+\Sym^jC$ under $\phi_k$ is  $(D) + \Theta_j$. Moreover, the restriction of $\phi_k$  is birational onto its image, and thus
$\phi_{k*}[D+\Sym^jC] = [\Theta_j]$, from which the proof follows.

\end{proof}

We now prove Theorem \ref{appendix}. First, we pick a translate $x+\Theta_{g-1}$ such that $v-x=(D)\in\Theta_{g-1}$ where $D$ is an ordinary divisor, and the dimension of $(V-x)\cap \Theta_{g-1}$ is $j-1$.

Next define $W=(V-x)\cap\Theta_{g-1}$, and set $W'$ to be the irreducible component of $\phi_k^* W$ containing $\phi_k^{-1}(v-x)$, so that $W'$ has dimension $j$ and maps surjectively onto the irreducible component $W$ containing $(v-x)$. Finally, set $W'' = \psi_k(W')$. Now let $L_0\subset \PP^{g-1}$ be a linear space of codimension $j-1$ which intersects $W''$ in isolated points and passes through 
$\psi_k\circ\phi_k^{-1}(v-x)$. 
Then $\psi_{k*}\circ\psi_k^*L_0$ intersects $W$ in isolated points and passes through $v-x$, and therefore $x+\psi_{k*}\circ\psi_k^*L_0$ intersects $V$ at isolated points and passes
through $v$. The theorem now follows as in \citep[Proposition 3.25]{TS} since the contribution to the intersection is positive at all points, and the intersection multiplicity at $v$ is bounded below by the multiplicity of $V$ at $v$.

\end{proof}

 \bibliographystyle{plainnat}

\bibliography{references}

\end{document}